\documentclass[a4,12pt]{amsart}
\usepackage{amsmath,amssymb,bbm,tikz,bm,verbatim}
\usetikzlibrary{arrows}
\usepackage{stmaryrd}
\usepackage{fnpct}
\newcommand{\mc}[1]{\mathcal{#1}}
\newcommand{\bs}[1]{\boldsymbol{#1}}

\newcommand\N{\mathbb N}
\newcommand\sF{\mathcal F}
\newcommand\sS{\mathcal S}
\newcommand\sG{\mathcal G}
\newcommand\sC{\mathcal C}
\newcommand\sL{\mathcal L}
\newcommand\q{\quad}
\newcommand\qq{\qquad}
\newcommand\es{\varnothing}
\newcommand\lam{\lambda}
\newcommand\La{\Lambda}
\newcommand\bmu{{\bm{\mu}}}
\newcommand\bal{{\bm{\alpha}}}
\newcommand\al{\alpha}

\newcommand\eps{\epsilon}

\newcommand\gcev[1]{{\overset{{}_\shortleftarrow}{#1}}}
\newcommand\gvec[1]{{\overset{{}_\shortrightarrow}{#1}}}

\long\def\blue#1{{\color{blue}{#1}}}

\RequirePackage[colorlinks,citecolor=blue,urlcolor=blue]{hyperref}
\usepackage{tikz}
\usetikzlibrary{arrows.meta}
\usepackage{todonotes}
\usetikzlibrary{shapes}
\usepackage{color}

\renewcommand{\P}{\mathbb{P}}
\newcommand{\E}{\mathbb{E}}
\newcommand\FF{\mathcal{F}} 
\newcommand{\Q}{\mathbb{Q}}
\newcommand\sP{\mathcal{P}}

\newcommand{\blank}[1]{}

\newcommand{\vep}{\eps} 
\newcounter{thmcounter}
\newtheorem*{theorem*}{Theorem}
\newtheorem{theorem}[thmcounter]{Theorem}
\newtheorem{lemma}[thmcounter]{Lemma}

\newtheorem{proposition}[thmcounter]{Proposition}
\newtheorem{question}[thmcounter]{Question}
\newtheorem{open}[thmcounter]{Open problem}

\theoremstyle{remark}
\newtheorem{definition}[thmcounter]{Definition}
\newtheorem{example}[thmcounter]{Example}
\newtheorem{remark}[thmcounter]{Remark}

\newcommand\ol{\overline}

\newcommand\oo{\infty}
\newcommand\supp{\text{\rm supp}}
\def\resp{respectively}
\newcommand\hlf{\frac12}
\newcommand\Leb{\Q} 
\newcommand\kmax{k_{\text{\rm max}}}
\renewcommand\o{\text{\rm o}}

\newcounter{mycount1}\newcounter{mycount2}\newcounter{mycount3}\newcounter{mycount}
\newenvironment{romlist}{\begin{list}{\rm(\roman{mycount1})}%
   {\usecounter{mycount1}\labelwidth=1cm\itemsep 0pt}}{\end{list}}

\newenvironment{numlist}{\begin{list}{\rm\arabic{mycount2}.}%
   {\usecounter{mycount2}\labelwidth=1cm\itemsep 0pt}}{\end{list}}
\newenvironment{letlist}{\begin{list}{\rm(\alph{mycount3})}%
   {\usecounter{mycount3}\labelwidth=1cm\itemsep 0pt}}{\end{list}}

\newcommand\sm{\setminus}
\newcommand\what{\widehat}
\newcommand\bpi{{\boldsymbol\pi}}
   
\numberwithin{equation}{section}
\numberwithin{thmcounter}{section}
\numberwithin{figure}{section}

\newcommand\oI{\ol I}

\newcommand{\END}{\hfill $\blacktriangleleft$}
\newcommand\mind{\mu_{\text{\rm ind}}}
\newcommand\sFp{\mc{F}^{\text{\rm perm}}}

\title{Coalescence in Markov chains}

\author{Geoffrey R.\ Grimmett and Mark Holmes}

\address{(GRG) Statistical Laboratory, Centre for
Mathematical Sciences, Cambridge University, Wilberforce Road,
Cambridge CB3 0WB, UK} 
\email{grg@statslab.cam.ac.uk}
\urladdr{\url{http://www.statslab.cam.ac.uk/~grg/}}

\address{(MH) School of Mathematics \&\ Statistics, The University of Melbourne, 
Parkville, VIC 3010, Australia}
\email{holmes.m@unimelb.edu.au}
\urladdr{\url{https://researchers.ms.unimelb.edu.au/~mholmes1@unimelb/}}

\begin{document}

\begin{abstract}
A Markov chain $X^i$ on a finite state space $S$ has transition matrix
$P$ and initial state $i$. We may run the chains $(X^i: i\in S)$
in parallel, while insisting that any two such chains coalesce whenever they are simultaneously at the same state. 
There are $|S|$ trajectories which evolve separately, but not necessarily independently,
prior to coalescence.
What can be said about the number $k(\mu)$ of coalescence classes of the process, and 
what is the set $K(P)$ of such numbers $k(\mu)$, as the coupling $\mu$ of the chains ranges
over  couplings that are consistent with $P$?
We continue earlier work of the authors (\lq Non-coupling from the past', in \emph{In and Out of Equilibrium 3}, 
Springer, 2021) on these two fundamental questions, which have special importance 
for the \lq coupling from the past' algorithm.

We concentrate partly on a family of  couplings termed block measures, which may be viewed as
couplings of lumpable chains with coalescing lumps. Constructions of such couplings
are presented, and also of non-block measure with similar properties.
\end{abstract}

\date{15 October 2025, revised 8 March 2026} 

\dedicatory{In celebration of Svante Janson's 70th birthday}
\maketitle

\section{Introduction}\label{sec:1}

Finite-state space Markov chains form a key topic in probability theory, and 
they are taught in 
many undergraduate courses worldwide. They are considered to be well understood,
and their theory fully established. Applications across science are significant and multifarious.
In this paper we explore some interesting questions concerning coalescence 
that remain unresolved, and indeed in  part unasked.

Throughout this paper, $S$ is a finite, non-empty set (without loss of generality, we may take 
$S=\{1,2,\dots, n\}$), and $P=(p_{i,j}:i,j\in S)$ is an irreducible stochastic matrix.
Let $X=(X_t: t = 0,1,\dots)$ be a Markov chain on 
the state space $S$ with transition matrix $P$. 
The long-term behaviour of $X$ is well known: if $X$ is (irreducible and)  aperiodic, then $P$ has 
a unique invariant distribution $\bpi=(\pi_j: j \in S)$, and $\P(X_t=j) \to \pi_j$
for $j\in S$, as $t\to\oo$. 
For introductions to the theory of Markov chains, the reader is referred to the books
\cite{GS,GW,Norris}.

For $i\in S$, denote by $X^i=(X_t^i: t=0,1,\dots)$ the Markov chain $X=(X_t)$ conditioned on starting in state $i$, that is, conditioned on the event that
$X_0=i$. In the current article we explore the simultaneous evolution of the chains 
$X^i$, $i\in S$, in the setting where any two paths stay together after they meet; 
that is, if $X^i_t=X^j_t$ then $X^i_{t+s}=X^j_{t+s}$ for all $s\in \N :=\{1,2,3,\dots\}$.  
For general $P$, multiple joint distributions of evolutions exist that are `consistent' with $P$ in
the sense that they have the required marginal distributions.  

In this setting, we say that $X^i$ and $X^j$ \emph{coalesce} (or just that $i$ and $j$ coalesce) if they meet 
(that is, if $X^i_t=X^j_t$ for some $t$).
The principal issue investigated here is to determine 
the degree of coalescence of the family $(X^i: i\in S)$ as $t\to\oo$.
This question is made more explicit in Section \ref{sec:cftp}, 
and is connected to earlier work on so-called avoidance coupling;
see Remark \ref{rem:kmax}.

The current work is a development of earlier work of the authors, \cite{ncftp},
directed mainly at coalescence in so-called \lq coupling from the past' (CFTP).
CFTP is an important technique for exact simulation from a given distribution on a state space $S$, and is related to the topic of \lq Monte Carlo Markov chains' (MCMC). Whereas MCMC runs forwards in time, CFTP
runs backwards. The CFTP
algorithm is successful in situations where all the chains $(X^i)$ coalesce. 
We explain this connection in Remark \ref{rem:fb}. 
Some background information from  \cite{ncftp} is included in this work as an aid to the reader. 

Here is a summary of the contents of this article. A \lq grand coupling' is
a coupling of the chains $(X^i: i \in S)$. Such couplings are defined in Section \ref{sec:grand},
where the special \lq independence coupling' is introduced. In Theorem \ref{thm:multiple_consistent} 
we answer the question of when the independence coupling is the unique grand coupling. 
Section \ref{sec:cftp} is devoted to forward coalescence.
Lumpable chains and block measures are introduced in Section \ref{sec:3}, and the
associated main results appear in Sections \ref{sec:exist} and \ref{sec:exist-non}.
A condition for a lumpable chain to give rise to a block measure is presented in Theorem \ref{prop:4};
conversely, the existence of non-block measures is explored in Theorem \ref{thm:6} in the special case of the 
transition matrix with equally probable transitions.  In the final Section \ref{sec:whatF}, we initiate an investigation of the structure
of the set of grand couplings corresponding to transition matrices. 

The general questions approached here seem quite fundamental 
to a complete theory of
finite state-space Markov chains,
and yet they have received only limited attention so far.  This is reflected in the partial
state of knowledge revealed in \cite{ncftp} and the current paper, which concentrate on grand couplings
based on \emph{sequences of independent random functions}.
Thus, this work may be viewed as a contribution to the
theory of iterated random functions. 
The idea of studying a \lq forward' chain by running it `backwards' appeared
in \cite{Letac}; see also \cite{CL, DP}, for example.

There are a number of questions and open problems presented in this paper, including Questions 
\ref{q:0}, \ref{q:1}, \ref{q:2}, \ref{q:4}, and the authors hope that readers will resolve some of these.

We write $\N$ for the set $\{1,2,\dots\}$
of natural numbers, and $\P$ for a 
generic probability measure.

\section{Grand couplings of Markov chains}\label{sec:grand}

We shall consider only \emph{irreducible} Markov chains on a finite state space $S$.
Let $\FF_S$ be the set of functions from $S$ to $S$, and
let $\sP_S$ be the set of irreducible  stochastic matrices on $S$. 

\begin{definition}[\cite{ncftp}]\label{def:consistent}
A probability measure $\mu$ on $\FF_S$ is \emph{consistent} with $P\in\sP_S$,
in which case we say that the pair $(P,\mu)$ is \emph{consistent},  if
\begin{equation}\label{eq:4}
p_{i,j} = \mu\bigl(\{f\in\FF_S: f(i)=j\}\bigr), \qq  i,j\in S.
\end{equation}
Let $\sL_P$ denote the set of probability measures $\mu$ on $\FF_S$ 
that are consistent with $P\in\sP_S$.  \END
\end{definition}

Let $P\in\sP_S$.
Any probability measure $\mu\in\sL_P$ may be used as the basis of a coupling of the
chains $(X^i: i\in S)$, as we now demonstrate. 
Let $F_1,F_2,\dots$ be random functions that are independent  and distributed
on $\sF_S$ according to $\mu$, and set 
\begin{equation}\label{eq:fcoupl}
X^i_t=F_t \circ F_{t-1}\circ\cdots\circ F_1(i), \qq i\in S.
\end{equation}
(This is a convenient abuse of notation;  more properly, the right side of \eqref{eq:fcoupl} defines  processes
having the same law as the sequence $(X^i)$.)
For $t\in\{0,1,\dots\}$, let $Z_t=(X^i_t: i \in S)$.  Then $Z=(Z_t:t=0,1,\dots)$ is a 
Markov chain with state space $S^S$, 
started in the state $(1,2,\dots,n)$.
The \lq multichain' $Z$ is in general reducible, since the number of distinct entries in $Z_{t+1}$ 
is less than or equal to that in $Z_t$.  

By the consistency property \eqref{eq:4}, 
the sequence $X^i=(X_t^i: t= 0,1,\dots)$
has the same law as $X$ conditioned on the event that $X_0=i$.  
Moreover, if $X_t^i=X_t^j$ then trivially $X_{t+s}^i=X_{t+s}^j$ 
for all $s\in \N$ (since each $F_t$ is a function);
that is, once the paths from $i$ and $j$ meet, they stay together thenceforth.

A measure $\mu\in\sL_P$ is sometimes called a \emph{grand coupling} of
$P$.
It is elementary that $\mc{L}_P\ne\es$, as indicated in the following example.

\begin{example}[Independence coupling]
\label{exa:indep_grand}
Let $P\in \mc{P}_S$, and let $\mind$ be given by the product form
\[
\mind(\{f\})=\prod_{i\in S}p_{i,f(i)}, \qq f\in \mc{F}_S.
\]
Then 
\begin{align*}
\mind(\{f:f(i)=j\})&=\sum_{f:f(i)=j}\prod_{u \in S}p_{u,f(u)}
=p_{i,j}\sum_{f:f(i)=j}\prod_{u \ne i}p_{u,f(u)}\\
&=p_{i,j}\sum_{(j_1,\dots, j_{i-1},j_{i+1},\dots, j_n)\in S^{n-1}}\prod_{u\ne i}p_{u,j_u}\\
&=p_{i,j}\prod_{u \ne i}\biggl[\sum_{j_u\in S}p_{u,j_u}\biggr]
=p_{i,j}\prod_{u\ne i}1,
\end{align*}
so that $\mind\in \mc{L}_P$.
The measure $\mind$ is called the \emph{independence coupling} as it gives rise to    
$n=|S|$ chains $X^i$ with transition matrix $P$, starting from each $i \in S$ respectively, 
that evolve independently until they meet.  
If $P$ is aperiodic then, by Doeblin's theorem (see Theorem \ref{thm:indep_copies_aperiodic}), 
for every pair $i,j\in S$ the chains $X^i$, $X^j$ meet in finite time. Such chains stick together thenceforth, and thus 
all $n$ chains a.s.\ meet in finite time (in
some common random state).
\END
\end{example}

The next theorem contains a criterion for $\mind$ to
be the unique grand coupling.

\begin{theorem}
\label{thm:multiple_consistent}
For $P\in\sP_S$, we have $|\mc{L}_P|\ge 2$ if and only if $P$ has at least two rows each 
of which contains some entry lying in the open interval $(0,1)$.
\end{theorem}

\begin{proof}
It can be useful to think in terms of the transition diagram of the chain, i.e., the labelled directed graph $G$
with vertex-set $S$, and with a  directed edge from any $i$ to any $j$ such that $p_{i,j}>0$;
such an edge is labelled with the value $p_{i,j}$. The \emph{in-degree} (\resp, \emph{out-degree}) of a 
state is the number of edges directed towards it
(\resp, away from it).

Let $R$ be the set of rows that contain some entry in the interval $(0,1)$. Since $P$ is stochastic, 
such a row must contain at least two such entries. 

(a)  Let $|R|\le 1$; we will show that the independence coupling   is the unique grand coupling.
We explain first the case $|R|=0$, for which every row contains a single non-zero entry $1$,
and thus all edge-labels in $G$ are $1$.
Each state $i$ has out-degree $1$, and we write $[i,t_i\rangle$ for the unique directed edge leading from $i$. 
There are exactly $n=|S|$ edges in $G$.
Since $P$ is irreducible, $G$ is a directed self-avoiding cycle of length $n$.
Therefore, $\mc{L}_P$ contains only the probability measure $\mu=\delta_f$ that is 
supported on the single function $f$ defined by  
$f:i \mapsto t_i$ for  $i \in S$. In this trivial situation,  $\mu$ is the independence coupling
$\mind$.

Suppose that $|R|=1$, and assume without loss of generality that the row labelled $1$ is the unique row
containing an entry in $(0,1)$. Let $J=\{j: p_{1,j}\in(0,1)\}$. Edges of $G$ of the form $[1,j\rangle$
for $j\in J$ have labels in $(0,1)$ and all other edges
are labelled $1$. For $i\ne 1$, write $[i,t_i\rangle$ for the unique edge directed away from $i$.
We have
\begin{alignat}{5}
\P(X_1^1= j) &=p_{i,j}\q&&\text{for } j \in J,\\
\P(X_1^i = t_i) &=1\q&&\text{for } i\ne 1.
\end{alignat}
There is a unique probability measure $\mu$ consistent with the above, namely
\begin{equation}\label{eq:new7}
\mu(\{f\})= p_{1,f(1)} \prod_{i\ne 1} \delta_{f(i),t_i},\qq f\in\FF_S,
\end{equation}
where $\delta$ is the Kronecker delta. This is simply the independence coupling. Therefore, $|\sL_P|=1$.

(b) Conversely, suppose $|R|\ge 2$, and pick two rows lying in $R$, say the $i$th and $j$th with $i\ne j$. Find 
$r$, $s$ such that $p_{i,r}p_{j,s}>0$. There is no loss of generality for the proof 
that follows if we assume
$i=1$ and $j=2$, and we assume this henceforth. We may assume further that
\begin{equation}\label{eq:pdiff}
0<p_{2,s}\le p_{1,r}<1.
\end{equation}   
Before defining
$\mu$ explicitly, we describe the stochastic evolution of the first time-step of the family $(X^i: i\in S)$
of Markov chains.

For $i\in\{2,3,\dots,n\}$, we let $X_1^i$ be chosen according to $P$, that is,
\begin{equation}\label{eq:new3}
\P(X_1^i=j)=p_{i,j}, \qq j \in S.
\end{equation}
Furthermore, $Z:=\{X_1^i: i = 2,3,\dots\}$ is a set of independent 
random variables.

We turn to $X_1^1$, which we take to be
independent of $Z\sm \{X_1^2\}$. Furthermore,
\begin{align}
\P(X_1^1=r \mid  X_1^2=s) &=1, \label{eq:new4}\\
\P(X_1^1=j \mid  X_1^2 \ne s) &= 
\begin{cases} 
\dfrac{p_{1,r}-p_{2,s}}{1-p_{2,s}} &\text{if } j=r,\\
\dfrac{p_{1,j}}{1-p_{2,s}} &\text{if } j\ne r.
\end{cases}\label{eq:new5}
\end{align}
It is an exercise in conditional probability that 
the mass function of $X_1^1$ is $(p_{1,j}: j\in S)$. 
Equations \eqref{eq:new4}--\eqref{eq:new5} amount to a coupling 
of $X_1^1$ and $X_1^2$ under which 
\eqref{eq:new4} holds.  Since $p_{1,r}<1$, 
\eqref{eq:new4} fails by \eqref{eq:pdiff} under the independence coupling
$\mind$. Therefore, $\mu\ne \mind$ and hence
$|\sL_P|\ge 2$.

Equations \eqref{eq:new3}--\eqref{eq:new5} may be expressed by defining $\mu$ as follows: 
\begin{equation}\label{eq:new6}
\mu(\{f\})=
\begin{cases}
p_{2,s}\prod_{i \ne 1,2}p_{i,f(i)} & \text{ if }f(1)=r, \,f(2)=s,\\
\dfrac{p_{1,r}-p_{2,s}}{1-p_{2,s}}\prod_{i \ne 1 }p_{i,f(i)} & \text{ if }f(1)=r,\, f(2)\ne s,\\
\dfrac{1}{1-p_{2,s}}\prod_{i}p_{i,f(i)} & \text{ if }f(1)\ne r,\, f(2)\ne s,\\
0& \text{ otherwise}.
\end{cases}
\end{equation}

One may check directly that $\mu\in\sL_P$, but it is quicker to verify the probabilities implied by
\eqref{eq:new3}--\eqref{eq:new5}.
\end{proof}

\begin{remark}[Random transition matrix]\label{rem:Q}
By Theorem \ref{thm:multiple_consistent}, a \lq typical' transition matrix $P$ has multiple grand couplings.    We make this statement more precise as follows.   For given $S$, we model a \lq typical' transition matrix as the  $|S|\times|S|$ matrix with elements
\begin{equation*}
p_{i,j} = q_{i,j}/Q_i,
\end{equation*}
where the $q_{i,j}$ are independent  and uniformly distributed on $(0,1)$, and $Q_i=\sum_j q_{i,j}$. 
Let $\Leb$ denote the law of such $P$.  Note that $\Q$-a.e.\ $P$ has all entries in $(0,1)$ and hence is irreducible and aperiodic.  Moreover, by Theorem \ref{thm:multiple_consistent}, $\Q$-a.e.\ $P$ admits multiple grand couplings.
Further results for a random transition matrix may be found in Remark \ref{rem:Pqn},
Lemma \ref{lem:10},  and Section
\ref{sec:whatF}. The particular measure $\Q$ has been chosen for simplicity; 
the given applications hold for more general measures, including those generated from 
independent $q_{i,j}$ with a common, strictly positive, density function on $(0,1)$.
\END\end{remark}

We will make frequent use of the following well known fact due to Doeblin \cite{Doeblin} (see also \cite[p.\ 260]{GS}).

\begin{theorem}[\cite{Doeblin}]
\label{thm:indep_copies_aperiodic}
Let $(X^i:i \in S)$ be \emph{independent} Markov chains on a finite state space $S$ with common irreducible, aperiodic transition matrix $Q$, with $X^i_0=i$ a.s.~for each $i\in S$.    
For $i,j\in S$, there exists a.s.\ a finite time $T$ such that $X^i_T=X^j_T$.
\end{theorem}

The following classical result of  Birkhoff and von Neumann will be useful later.

\begin{theorem}[\cite{Birkh,vN}]\label{birk}
A stochastic matrix $P$ on the finite state space $S$ is doubly stochastic
if and only if it lies in the convex hull of the set $\Pi_S$ of permutation matrices.
\end{theorem}

\section{Coalescence of trajectories}\label{sec:cftp}

Consider a Markov chain $X$ on the state space $S=\{1,2,\dots,n\}$ with
irreducible transition matrix $P$, and let $\mu\in\sL_P$ be a grand coupling. 
We consider the coalescence of the chains $X^i$ in this section,
and begin with some notation from \cite{ncftp}.

As in Section \ref{sec:grand} (see \eqref{eq:fcoupl}), 
let $F_1,F_2,\dots$ be independent random functions distributed
according to $\mu$, and set 
$X^i_t=\gvec F_t(i)$ where $i\in S$ and
\begin{equation}\label{eq:defvF}
\gvec F_t := F_t \circ F_{t-1}\circ\cdots\circ F_1.
\end{equation}
Then  $(X^i:i\in S)$ is
a family of coupled Markov chains running forwards in time, each having transition matrix $P$,
and such that $X^i$ starts in state $i$.

For $i,j\in S$, let 
\begin{equation}\label{eq:co}
T_{i,j}=\inf\{t:X^i_t=X^j_t\}.
\end{equation}
We say that $i$ and $j$ \emph{coalesce} (and write $i \sim j$)  if $T_{i,j}<\oo$.  
The \emph{forward coalescence time}
is given by
\begin{align}
T&=\inf\bigl\{t:    \gvec F_t(\cdot)\text{ is a constant function}\bigr\}
= \sup_{i,j\in S}T_{i,j}.\label{eq:cotime2}
\end{align}
We say that \emph{coalescence occurs} if $\P(T<\oo)=1$.

\begin{lemma}\label{lem:eq}
The relation $\sim$ is a (random)  equivalence relation on $S$.
\end{lemma}

\begin{proof}
It suffices to prove the transitivity of $\sim$. Let $i\sim j$ and $j\sim k$.
For $t\ge \max\{T_{i,j}, T_{j,k}\}$ we have that $X_t^i = X_t^j = X_t^k$. The claim follows.
\end{proof}

The equivalence classes of $\sim$ are termed the \emph{coalescence classes} of 
$(X^i:i \in S)$.  
The number of coalescence classes is a.s.~constant (see \cite[Lemma 2]{ncftp}),
and we denote this number by $k(\mu)$ and call it the \emph{coalescence number} of $\mu$. 
We define
\begin{equation}\label{eq:cono}
K(P)=\{k(\mu): \mu\in \sL_P\}.
\end{equation}

\begin{remark}[Coupling from the past (CFTP)]\label{rem:fb}
CFTP is a prominent algorithm
introduced by Propp and Wilson \cite{PW2,PW3,WP1} for perfect simulation 
from  the invariant distribution of an 
irreducible, aperiodic Markov chain on a finite state space.  The CFTP algorithm may be defined 
by replacing the function $\gvec F_t$ in \eqref{eq:defvF} by the function 
$\gcev F_t := F_1 \circ F_{2} \circ \dots \circ F_t$.
The \emph{backward coalescence time} is defined by
\begin{equation}\label{eq:cotime}
C=\inf\bigl\{t:    \gcev F_t(\cdot)\text{ is a constant function}\bigr\},
\end{equation}
and \emph{backward coalescence} is said to  occur if $\P(C<\oo)=1$.
On the event $\{C<\oo\}$, $\gcev F_C$ may be regarded
as a random state, and the main theorem of CFTP asserts that, if backward coalescence occurs, then
$\gcev F_C$ is distributed as the invariant distribution of the transition matrix $P$.

The relationship between forward and backward coalescence was explored in \cite{ncftp}.  
It turns out that $C$ 
and $T$ are identically distributed. Furthermore, the CFTP process has the same set $K(P)$ of coalescence numbers
as the forward process. 
In contrast, there is a significant difference between forward and backward coalescence in that,  
if $\gvec F_t(i)=\gvec F_t(j)$, then $\gvec F_{t+1}(i)=\gvec F_{t+1}(j)$, whereas 
the corresponding statement for backward coalescence is false in general.  
The last occurs whenever the coalescing classes in the forward direction are 
non-deterministic (see, for example, 
the forthcoming Examples \ref{ex:7} and \ref{ex:sub_blocks}(b)).
\END\end{remark}

The following two questions are fundamental to understanding coalescence.

\begin{question}\label{q:0}
\mbox{\hfill}
\begin{numlist}
\item
Can we determine the set $K(P)$ for given $P$?
\item
Which $\mu\in\sL_P$ have $k(\mu)=1$?
\end{numlist}
\end{question}
We shall see in Theorem \ref{lem:8} that the independence coupling
$\mind$ of Example \ref{exa:indep_grand} satisfies $k(\mind)\le k(\mu)$
for all $\mu\in\sL_P$, and moreover $k(\mind)=1$
if and only if $P$ is aperiodic.

\emph{Henceforth, expressions involving the word \lq coalescence'  
shall refer to \emph{forward} coalescence.}
Let $\mu$ be a probability measure on $\FF_S$, 
and let $\supp(\mu)$ denote the support of $\mu$.
Let 
$F=(F_s:s \in \N)$ be a vector of independent and identically distributed 
random functions, each with law $\mu$.  The law of $F$ is the product measure
$\bmu=\prod_{i\in \N}\mu$. 

\begin{remark}[Avoidance coupling]\label{rem:kmax}
Let $P\in\sP_S$ and let 
\begin{equation*}
\kmax=\kmax(P):= \max\{k(\mu): \mu\in\sL_P\}.
\end{equation*}
That is, $\kmax$ is the maximum $k$ such that: there exists
some grand coupling $\mu$ for which there exist $k$ (possibly random) 
initial states whose trajectories avoid one 
another for all time.
The identification
of $\kmax$  might be termed the \lq avoidance problem with simultaneous updating'.
The related problem of avoidance coupling with \emph{sequential} updating was initiated
in \cite{AHMWW} for random walk on a graph and has been developed further by others
(see, for example, \cite{BP}). 
\END\end{remark}

Recall that, since $P$ is finite and irreducible, it has a \emph{unique} invariant distribution
(see, for example, \cite[Thm 6.4.3]{GS}).

\begin{theorem}\label{thm:suffc}
Let $P\in\sP_S$ have (unique) invariant distribution $\bpi$.
\begin{letlist}
\item Let $m\in \N$, and suppose there exists $s\in S$ with $\pi_s>1/m$. Then
$\kmax< m$.
\item
Let $m\in \N$, and suppose there exists $s\in S$ such that $p_{i,s} > 1/m$ for all $i\in S$.
Then $\kmax< m$.
\end{letlist}
\end{theorem}

\begin{proof}
(a) Assume $\pi_s> 1/m$, and let $i_1, i_2,\dots,i_m$ be distinct elements of $S$. 
For each $k\in\{1,2,\dots,m\}$,
there is asymptotic density $\pi_s$ of times $n$ at which $X_n^{i_k}=s$.
Since $\pi_s>1/m$, there exist (a.s.) distinct $i,j\in\{i_1,i_2,\dots,i_m\}$ such that 
there is a strictly positive density of times $n$ at which $X_n^i=X_n^j=s$. 
Therefore, $k(\mu)<m$ for all $\mu\in \sL_P$.

(b) Assume the given condition holds. Then
\begin{equation}\label{eq:Nsuffc}
\pi_s=\sum_{i\in S} \pi_i p_{i,s} > \frac1m.
\end{equation}
The conclusion holds by part (a).
\end{proof}
\begin{remark}\label{rem:suffc}
The condition of part (b) may be changed slightly,
as follows. Suppose $p_{i,s} > 1/m$ for all $i\in S$ with $i\ne s$. Then \eqref{eq:Nsuffc} becomes 
\begin{equation*}
\pi_s \ge \sum_{i\ne s} \pi_i p_{i,s}
> \frac{1}{m}(1-\pi_s).
\end{equation*}
Therefore, $\pi_s>1/(m+1)$,
whence $\kmax< m+1$ by part (a).
\END\end{remark}

\begin{remark}[Random transition matrix]\label{rem:Pqn}
By Theorem \ref{thm:suffc}, if there exists $s\in S$ with $\pi_s>\frac12$,
then $k(\mu)=1$ for all $\mu\in \sL_P$.  In particular, there is strictly positive $\Leb$-probability
(see Remark \ref{rem:Q}) 
that a random transition matrix $P$ satisfies $K(P)=\{1\}$.
\END\end{remark}

\begin{open}
\label{q:1}
Is it true that $K(P)=\{1\}$ for $\Leb$-a.e.~$P$?
\end{open}
By Lemma \ref{lem:8}(a) below, we have that $1 \in K(P)$ for $\Leb$-a.e.\ $P$.

Although $k(\mu)$ is a.s.\ constant,
the coalescence classes of $\sim$ need not themselves be a.s.\ constant.
Here is an example of this, preceded by some notation.

\begin{definition}\label{def:fn}\mbox{\hfill}
Let $f \in \sF_S$ where $S=\{1,2,\dots,n\}$. We write $f=(j_1j_2\ldots j_n)$ if
$f(r)=j_r$ for $r=1,2,\dots,n$.\END
\end{definition}

\begin{example}\label{ex:7}
Take $S=\{1,2,3,4\}$ and any 
	consistent pair $(P,\mu)$
	with 
	\begin{equation*}
	\supp(\mu)=\{f_{1,1},f_{1,2},f_{2,1},f_{2,2}\}
	\end{equation*}
 where
	\begin{equation*}
f_{1,1}=(1212),\q f_{1,2}=(1221),\q f_{2,1}=(3434), \q f_{2,2}=(3443).
		\end{equation*}
Then $k(\mu)=2$ but the coalescence classes
of $\gvec F$ may be either $\{1,3\}$, $\{2,4\}$ or $\{1,4\}$, $\{2,3\}$,
each having a strictly positive probability.
The value of $F_1$ determines which of these two possibilities 
occurs\footnote{In the case of backward coalescence, for $r\in \{1,2\}^2$,
the `first' function to be applied is $f_r$ a.s.\ infinitely often,
	whence  $\gcev F_t(1)=  \gcev{F}_t(3)$ infinitely often and 
	$\gcev{F}_t(1)\ne   \gcev{F}_t(3)$ infinitely often.}.	The four functions $f_{i,j}$ are illustrated in Figure \ref{random_coalesce}. 
	\END
\end{example}

\def\x{3.8}
\def\y{7.6}
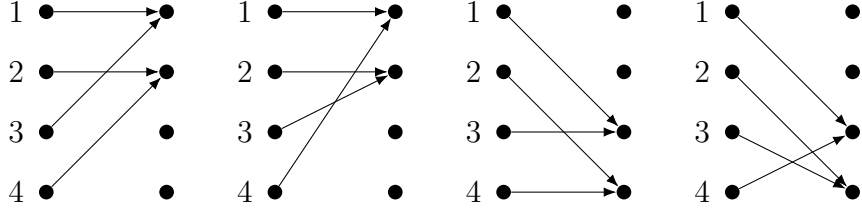
\begin{figure}
	\begin{center}
		\begin{tikzpicture}[scale=.8]
		\node (AA1) at (-.5,0) {1};
		\node (AA2) at (-.5,-1) {2};
		\node (AA3) at (-.5,-2) {3};
		\node (AA4) at (-.5,-3) {4};
		\node[scale=0.5,circle,fill=black] (A1) at (0,0) {};
		\node[scale=0.5,circle,fill=black] (A2) at (0,-1) {};
		\node[scale=0.5,circle,fill=black] (A3) at (0,-2) {};
		\node[scale=0.5,circle,fill=black] (A4) at (0,-3) {};
		\node[scale=0.5,circle,fill=black] (B1) at (2,0) {};
		\node[scale=0.5,circle,fill=black] (B2) at (2,-1) {};
		\node[scale=0.5,circle,fill=black] (B3) at (2,-2) {};
		\node[scale=0.5,circle,fill=black] (B4) at (2,-3) {};
		\draw[-Latex] (A1)--(B1);
		\draw[-Latex] (A2)--(B2);
		\draw[-Latex] (A3)--(B1);
		\draw[-Latex] (A4)--(B2);

		\node (CC1) at (-.5+\x,0) {1};
		\node (CC2) at (-.5+\x,-1) {2};
		\node (CC3) at (-.5+\x,-2) {3};
		\node (CC4) at (-.5+\x,-3) {4};
		\node[scale=0.5,circle,fill=black] (C1) at (0+\x,0) {};
		\node[scale=0.5,circle,fill=black] (C2) at (0+\x,-1) {};
		\node[scale=0.5,circle,fill=black] (C3) at (0+\x,-2) {};
		\node[scale=0.5,circle,fill=black] (C4) at (0+\x,-3) {};
		\node[scale=0.5,circle,fill=black] (D1) at (2+\x,0) {};
		\node[scale=0.5,circle,fill=black] (D2) at (2+\x,-1) {};
		\node[scale=0.5,circle,fill=black] (D3) at (2+\x,-2) {};
		\node[scale=0.5,circle,fill=black] (D4) at (2+\x,-3) {};
		\draw[-Latex] (C1)--(D1);
		\draw[-Latex] (C2)--(D2);
		\draw[-Latex] (C3)--(D2);
		\draw[-Latex] (C4)--(D1);
		
		\node (EE1) at (-.5+\y,0) {1};
		\node (EE2) at (-.5+\y,-1) {2};
		\node (EE3) at (-.5+\y,-2) {3};
		\node (EE4) at (-.5+\y,-3) {4};
		\node[scale=0.5,circle,fill=black] (E1) at (0+\y,0) {};
		\node[scale=0.5,circle,fill=black] (E2) at (0+\y,-1) {};
		\node[scale=0.5,circle,fill=black] (E3) at (0+\y,-2) {};
		\node[scale=0.5,circle,fill=black] (E4) at (0+\y,-3) {};
		\node[scale=0.5,circle,fill=black] (F1) at (2+\y,0) {};
		\node[scale=0.5,circle,fill=black] (F2) at (2+\y,-1) {};
		\node[scale=0.5,circle,fill=black] (F3) at (2+\y,-2) {};
		\node[scale=0.5,circle,fill=black] (F4) at (2+\y,-3) {};
		\draw[-Latex] (E1)--(F3);
		\draw[-Latex] (E2)--(F4);
		\draw[-Latex] (E3)--(F3);
		\draw[-Latex] (E4)--(F4);

		\node (GG1) at (-.5+3*\x,0) {1};
		\node (GG2) at (-.5+3*\x,-1) {2};
		\node (GG3) at (-.5+3*\x,-2) {3};
		\node (GG4) at (-.5+3*\x,-3) {4};
		\node[scale=0.5,circle,fill=black] (G1) at (0+3*\x,0) {};
		\node[scale=0.5,circle,fill=black] (G2) at (0+3*\x,-1) {};
		\node[scale=0.5,circle,fill=black] (G3) at (0+3*\x,-2) {};
		\node[scale=0.5,circle,fill=black] (G4) at (0+3*\x,-3) {};
		\node[scale=0.5,circle,fill=black] (H1) at (2+3*\x,0) {};
		\node[scale=0.5,circle,fill=black] (H2) at (2+3*\x,-1) {};
		\node[scale=0.5,circle,fill=black] (H3) at (2+3*\x,-2) {};
		\node[scale=0.5,circle,fill=black] (H4) at (2+3*\x,-3) {};
		\draw[-Latex] (G1)--(H3);
		\draw[-Latex] (G2)--(H4);
		\draw[-Latex] (G3)--(H4);
		\draw[-Latex] (G4)--(H3);
		\end{tikzpicture}
	\end{center}
	\caption{Diagrammatic representations of the four
	functions $f_{i,j}$ of Example \ref{ex:7}.} 
	\label{random_coalesce}
\end{figure}

\begin{open}
\label{q:2}
For given $P\in\sP_S$, is the set of cardinalities (possibly with repetition)
of coalescing classes a deterministic set? 
\end{open}

A probability measure $\mu$ on $\FF_S$ may be written in the form
\begin{equation}\label{eq:5}
\mu=\sum_{f\in\FF_S} \alpha_f \delta_f,
\end{equation}
where $\bal$ is a probability mass function on $\FF_S$ with support $\supp(\mu)$,
and $\delta_f$ is the Dirac delta-mass on the point  $f \in \FF_S$.  
Thus, $\alpha_f>0$ if and only if $f \in\supp(\mu)$.
We denote by $M(\sF_S)$ the set of all probability measures on $\sF_S$, noting that they may be represented
in the form \eqref{eq:5}.

In advance of stating an extension of \cite[Lemma 3]{ncftp}, we remind the reader of the definition of a cyclic class.

\begin{definition}\label{def:cyclic}
Consider an irreducible Markov chain on the state space $S$ with transition matrix $P$
and period $d$. There exists a unique partition
of $S$ into $d$ classes $S_0,S_1,\dots,S_{d-1}$ such that
\begin{equation*}
\sum_{j\in S_{r+1}} p_{i,j} = 1, \qq i\in S_r,\ r=0,1,\dots,d-1,
\end{equation*}
where by convention $S_0 = S_d$. The sets $S_r$ are called the \emph{cyclic classes} of the chain.\END
\end{definition}

See \cite[Sect.\ 2.3.2]{Bremaud} for further details of cyclic classes.

\begin{theorem}\label{lem:8}\mbox{\hfil}
\begin{letlist}
\item
We have that $1\in K(P)$ if and only if $P\in\sP_S$ is aperiodic. 
In this case, $k(\mind)=1$.
\item
If $P\in\sP_S$ has period $d$, then the coalescence classes of $\mind$
are a.s.\ the cyclic classes of $P$ and therefore  $k(\mind)=d$.
Hence, $d\in K(P)$, and moreover $d\le k$ for all $k\in K(P)$. 
\end{letlist}
\end{theorem}

Consider the particular case when all entries of $P$ are strictly positive. Since
$P$ is aperiodic, 
we have by Theorem \ref{lem:8}(a) that $1\in K(P)$. 
However, $K(P)$ may be larger than a singleton; see the
forthcoming Theorem \ref{thm:6}.

\begin{proof}[Proof of Theorem \ref{lem:8}]
(a) This is found at the proof of \cite[Lemma 3]{ncftp}, but is included here for clarity.
The independence coupling of Example \ref{exa:indep_grand} gives rise to $n=|S|$ chains with 
transition matrix $P$, starting from $1,2,\dots, n$,
respectively, that evolve independently until they meet.  
If $P$ is aperiodic (and, by assumption, irreducible) then all $n$ chains meet a.s.\ in finite time.  	
The last holds by Theorem \ref{thm:indep_copies_aperiodic}, since any paths that meet will remain together thenceforth
under this coupling.
	
Conversely, if $P$ is periodic and $p_{i,j}>0$ then $i\ne j$, and $i$ and $j$ can never coalesce, 
implying $1\notin K(P)$. 

(b) Suppose $P$ has period $d\ge 2$, and let $S_0,S_1,\dots,S_{d-1}$ be the cyclic classes of $P$ 
(see Definition \ref{def:cyclic}). For $r\ne s$,
states $i\in S_r$ and $j\in S_s$ do not coalesce for any $\mu \in \sL_P$. Now, $P^d$ has a block diagonal form with (transition)
matrices $E_0,E_1,\dots,E_{d-1}$
along its diagonal. Each $E_r$ is irreducible and aperiodic. By the argument above, subject
to the independence coupling  $\mind$, 
any two states in any given $S_r$ coalesce a.s. 

For clarity, we give some more details of the last step. Fix $r$ and let $i,j\in S_r$.
Let $Z^k=(Z_t^k: t=0,1,\dots)$ be a copy of $X^k$, and suppose 
the chains $(Z^k: k\in S_r)$
evolve independently (unlike the sequence $(X^k)$ whose members coalesce when they meet).
Now, $Z^{k,d} :=(Z^k_{md}: m= 0,1,\dots)$ has irreducible, aperiodic transition matrix $E_r$. 
By Theorem \ref{thm:indep_copies_aperiodic} 
there exists
a.s.~some time $T<\oo$ at which $Z^i_{Td}=Z^j_{Td}$. Therefore, there is some earliest time $U\le Td$
at which $Z^i_{U}=Z^j_{U}$.
From the pair $Z^i$, $Z^j$ we construct $X^i$, $X^j$ by 
$$
(X^i_t,X^j_t) = 
\begin{cases}
(Z^i_t,Z^j_t) &\text{for } t\le U,\\
(Z^i_t,Z^i_t) &\text{for } t> U.
\end{cases}
$$
Since the pair $(X^i,X^j)$ has the same law as under $\bmu$, this proves that $i$ and $j$ coalesce a.s.\ under $\bmu$.
This holds for all pairs of distinct states in $S_r$, and  
two states stick together after they coaleasce.
The second claim of the theorem is proved.

The minimality of $d$ holds since no two states in distinct cyclic classes may coalesce.
\end{proof}

\section{Lumpability and block measures}\label{sec:3}
The notion of lumpability was introduced in 1963 by Kemeny and Snell, \cite{KS}.
A Markov chain $X$ with finite state space $S$ is called lumpable  
if
there exists a partition $\sS$ of $S$ such that the projection of $X$ onto  
$\sS$ is itself a Markov chain. 

Here is a more precise definition.
A partition $\sS=\{S_1,S_2,\dots, S_\ell\}$ of $S$ is called \emph{trivial} if $|\mc{S}|\in \{1,|S|\}$, (i.e.,
if either  $\sS=\{S\}$ or $\sS=\{\{s\}: s\in S\}$). Elements of the partition $\sS$ are called
\emph{blocks}. 
The chain $X$ is called \emph{$\sS$-lumpable} if the sequence $(W_t:t=0,1,\dots)$, defined by 
\begin{equation}\label{eq:deflump}
W_t=j \q\text{if}\q X_t\in S_j,
\end{equation}
forms a Markov chain. The process $X$ is called \emph{lumpable} if
it is $\sS$-lumpable for some non-trivial partition $\sS$.

There is a limited literature on lumpable chains, for which the reader may consult, for example, \cite{geiger, gdpy, marin} and  \cite[Exer.\ 6.1.13]{GS}.
Kemeny and Snell \cite{KS} proved a necessary and sufficient condition for
$X$ to be $\sS$-lumpable, namely the following.

\begin{theorem}[\cite{KS}]\label{thm:KS}
Let $P=(p_{i,j}:i,j\in S)$.
Let $\sS$ be a partition of $S$, and let
\begin{equation}\label{eq:lump1}
\lambda_{r,s}^{(i)}= \sum_{j \in S_s} p_{i,j}, \q i \in S_r.
\end{equation}
A Markov chain $X$ with transition matrix $P$ 
 is $\sS$-lumpable if and only if, 
\begin{equation}\label{eq:lump2}
\text{for every $r$, $s$, we have that $\lambda_{r,s} := \lambda_{r,s}^{(i)}$ is constant 
for $i\in S_r$.}
\end{equation}
\end{theorem}

A stochastic matrix $P\in \mc{P}_S$ is called
$\mc{S}$-lumpable if \eqref{eq:lump2} holds (where the $\lambda_{r,s}^{(i)}$ are given 
in \eqref{eq:lump1}).  For such a pair $P$ and $\mc{S}$, 
let $\La$ be the $\ell\times\ell$ matrix $(\lambda_{r,s}: 1\le r,s\le \ell)$.
This $\La$ is the transition matrix of the \lq block process' $W$ of \eqref{eq:deflump}.

The evolution of an $\sS$-lumpable chain $X$ may be given in two stages. 
Suppose $X_t=i\in S_r$; then $X_{t+1}$ is given as follows.
\begin{letlist}
\item Select a random block $B$ with mass function $\P(B=S_s)=\lambda_{r,s}$. 
\item Conditional on $B=S_s$, choose $X_{t+1}$ with mass function
$$
\P(X_{t+1}=j\mid B=S_s)= p_{i,j}/\lambda_{r,s}.
$$ 
\end{letlist}

We address a certain sub-category of lumpable chains here, namely those for which $\La$
is doubly stochastic. By Theorem \ref{birk},
such $\La$ may be expressed as a convex combination of permutation matrices.
	
Recall from \eqref{eq:5} the set $M(\sF_S)$ of probability measures on $\sF_S$.

\begin{definition}[\cite{ncftp}]\label{def:block}
Let  $\mu\in M(\sF_S)$. For a partition 
$\sS = \{S_r: r\in I\}$ of $S$ with index set $I=\{1,2,\dots,\ell\}$, we call $\mu$ 
an \emph{$\sS$-block measure} if
\begin{letlist}
\item
	for $f \in \supp(\mu)$, 
	there exists a unique permutation $\pi=\pi_f$ of $I$
	such that, for $r\in I$,  $fS_r \subseteq S_{\pi(r)}$, and
\item $k(\mu)=\ell$.
\end{letlist}
An $\sS$-block measure $\mu$ is said to be \emph{trivial} if $\mc{S}$ is trivial.
A measure $\mu$ is
called a \emph{block measure} if it is an $\sS$-block measure for some partition $\sS$.\END
\end{definition}

Since any two states in distinct blocks cannot coalesce under a block measure, the condition $k(\mu)=\ell$ implies that 
\begin{equation}\label{new100}
\text{for $r\in I$ and $i,j\in S_r$, the pair $i$, $j$ coalesce a.s.,}
\end{equation}
so that the coalescence classes of the coalescence relation $\sim$ are a.s.\ the blocks
$S_1,S_2,\dots,S_\ell$.

By combining the definitions of lumpability and block measures we arrive at a necessary
condition for $\mu$ to be an $\sS$-block measure. The proof may be found in \cite{ncftp}.

\begin{theorem}[\mbox{\cite[Thm 6]{ncftp}}]\label{thm:7}
Let $S$ be a non-empty, finite set, let $P\in\sP_S$, and let $\sS = \{S_r: r\in I\}$  be  a partition of $S$
with index set
$I=\{1,2,\dots,\ell\}$. 
For $i\in S_r$, let
\begin{equation}\label{new:39+}
\lambda_{r,s}^{(i)}= \sum_{j \in S_s} p_{i,j}.
\end{equation}
If there exists an $\sS$-block measure $\mu\in \sL_P$, then
\begin{letlist}
\item
for $r,s\in I$, 
\begin{equation}\label{new:40-}
\text{$\lambda_{r,s} := \lambda_{r,s}^{(i)}$ is constant for $i\in S_r$},
\end{equation}

\item
the \lq block-transition matrix'  $\La=(\lambda_{r,s}: 1\le r,s\le \ell)$ is irreducible. Moreover, 
\begin{equation}\label{eq:ds}
\text{$\La$ is doubly stochastic,}
\end{equation}
which may be expressed as
\begin{equation}\label{new:43}
\sum_{i\in S}\sum_{j \in S_s} \frac1{|S_{r(i)}|}p_{i,j}=1, \qq r,s\in I,
\end{equation} 
where $r(i)$ is the index $r\in I$ such that $i \in S_r$.
\end{letlist}
\end{theorem}

Whether or not $\mu\in M(\sF_S)$ is a block measure turns out to depend on 
whether or not its coalescing classes are deterministic, that is,
a.s.~constant.
(Recall the definition of coalescing class after Lemma \ref{lem:eq}.)  We
state this as a theorem.

\begin{theorem}\label{thm:coal}
A probability measure $\mu\in M(\sF_S)$ is a block measure if and only
if its coalescing classes $\sC=\{C_1,C_2,\dots,C_\ell\}$ are a.s.~constant.   
If the last holds, $\mu$ is
a $\sC$-block measure.
\end{theorem}

\begin{proof}
If $\mu\in M(\sF_S)$ is an $\mc{S}$-block measure then (recall \eqref{new100} and the preceding discussion) the coalescing classes are a.s.~the elements of $\mc{S}$.

Conversely, let the coalescing classes be $C_1,\dots, C_\ell$, and assume they are a.s.~constant.  We claim
that Definition \ref{def:block} is satisfied with $\sS=\sC$.
  Clearly  $k(\mu)=\ell$, so (b) of Definition \ref{def:block} holds.  
  Let $f\in \supp(\mu)$, and let $i_1,i_2\in C_j$.  Since $i_1$ and $i_2$ a.s.~coalesce, $f(i_1)$ and $f(i_2)$ a.s.~coalesce, so $f(i_1)\in C_s\iff f(i_2)\in C_s$.  
Suppose instead that $i_1\in C_{j_1}$ and $i_2\in C_{j_2}$ where $j_1\ne j_2$.  This 
implies that $i_1$ and $i_2$ cannot coalesce, and hence 
$f(i_1)\in C_s$ implies $f(i_2)\notin C_s$.  

We have shown that $f$ permutes classes, which  verifies Definition \ref{def:block}(a) with $\sS=\sC$, 
and hence completes the proof.
\end{proof}

\begin{example}
\label{ex:sub_blocks}
Here is an illustration of Theorem \ref{thm:coal}. Consider the transition matrix 
\[P=\begin{pmatrix}
\hlf& 0 & \hlf& 0\\
0 & \hlf& 0 & \hlf\\
0 & \hlf & \hlf & 0\\
\hlf & 0 & 0 & \hlf
\end{pmatrix},\]
and note that it is irreducible and aperiodic.  
\begin{letlist}
\item
Let $\mu_f$ put mass  $\hlf$ on each of the two permutations
$$
f_1=(1234), \qq f_2= (3421).
$$
(Recall Definition \ref{def:fn}.)
It is easily checked that $\mu_f\in\sL_P$ is a (trivial) block measure  with blocks
$\{1\}$, $ \{2\}$, $\{3\}$, $\{4\}$, so that $k(\mu_f)=4$. 

\item
Let $\mu_g$ put mass $\frac12$ on each of the two functions
$$
g_1= (3434), \qq g_2=(1221).
$$
Then $\mu_g\in\sL_P$. Under $\mu_g$ either: $1$ and $3$ coalesce, \emph{and} $2$ and $4$ coalesce 
\emph{and there is no other coalescence} (this happens if $g_1$ is applied first) or: $2$ and $3$ coalesce, 
\emph{and} $1$ and $4$ coalesce \emph{and there is no other coalescence} (this happens if $g_2$ is applied first).  
Therefore, $k(\mu_g)=2$, but the coalescence sets of $\mu_g$ are not a.s.~constant. By
Theorem \ref{thm:coal},
$\mu_g$ is not a block measure.
This case contains the essence of Example \ref{ex:7}.\END
\end{letlist}
\end{example}

\section{Existence of block measures}\label{sec:exist}
This section is concerned with the existence of block measures.   The main question of interest is, for what 
state space $S$, partition $\sS$, and matrix 
$P\in \mc{P}_S$ does there exist an $\mc{S}$-block measure $\mu \in \mc{L}_P$?  
Let $\mc{C}=\mc{C}(\mu)$ denote the set of coalescing classes of $\mu$.  
In general $\mc{C}(\mu)$ is random, and the support of $\mc{C}(\mu)$ depends only on $\supp(\mu)$.   
By Theorem \ref{thm:coal}, finding an $\mc{S}$-block measure $\mu \in \mc{L}_P$  is equivalent to finding
 $\mu\in \mc{L}_P$ for which $\mc{C}=\mc{S}$ a.s.

By Theorem \ref{lem:8}, 
for every $P\in \sP_S$ there exists a partition $\sS$ of $S$ (comprising the cyclic classes
of $P$) such that the independence coupling $\mind$ is an
$\sS$-block measure.   
Recall the probability measure $\Leb$ in Remark \ref{rem:Q}.

\begin{lemma}\label{lem:10}
For $\Leb$-a.e.\  $P\in\sP_S$,
every block measure is trivial.  
\end{lemma}

\begin{proof}
We need only consider cases with $|S|\ge 3$.
If there exists a non-trivial block measure, then there exists $T\subseteq S$ 
such that $2\le |T| \le |S|-1$ and, by \eqref{new:39+}--\eqref{new:40-},
$$
\sum_{j\in T} p_{i,j} \text{ is constant for } i\in T.
$$
Let $Z_i:=\sum_{j\in T} p_{i,j}$. Then $Z_1,Z_2,\dots$ are independent with a common
absolutely continuous distribution.  Therefore, $\Leb(Z_i=Z_k)=0$ for all distinct pairs $i$, $k$,
 whence the set of such $P$ has $\Leb$-measure $0$.
\end{proof}

By Lemma \ref{lem:10}, for every $S$, for $\Leb$-a.s.\ $P\in \mc{P}_S$, and for every non-trivial $\mc{S}$,
 there exists no $\sS$-block measure. In contrast, the following holds.

\begin{theorem}
For $S=\{1,2,\dots,n\}$ 
and any partition $\mc{S}$ of $S$ there exists $\mu\in M(\sF_S)$ such that $\mu$ is an $\mc{S}$-block measure.
\end{theorem}

\begin{proof}
Fix $S$ and $\mc{S}=\{S_r: r\in I\}$.  Consider the set $\what\sF$ of functions $f$  such that 
\begin{romlist}
\item for all $r$, $f(i)=f(j)$ for every $i,j\in S_r$, and 
\item there exists a permutation $\pi=\pi_f$ of $I$
	such that, for $r\in I$,  $fS_r \subseteq S_{\pi(r)}$. 
\end{romlist}
For example, if $\mc{S}=\{\{1,2\},\{3,4,5\}\}$ then $\what\sF$ consists of the 12 functions denoted 
\begin{equation*}
\begin{aligned}
(11333), (11444),(11555),(22333), (22444),(22555),\\
(33111), (44111), (55111), (33222), (44222), (55222),
\end{aligned}
\end{equation*}
in the notation of Definition \ref{def:fn}.

Let $\mu$ be a probability measure with $\supp(\mu)=\what\sF$.  
By assumption (i), for all $r\in I$ and
$i,j\in S_r$, $i$ and $j$ coalesce in one step.
The claim follows by assumption (ii) and Definition \ref{def:block}. 
\end{proof}

Henceforth, for certain $P$, $\mc{S}$, we will present a natural measure
 $\mu\in \mc{L}_P$, and then determine conditions under which it is an $\mc{S}$-block measure.

Let the pair $P\in \mc{P}_S$ and $\sS=\{S_r: r\in I\}$ 
(a partition
of the state space $S$ with $I=\{1,2,\dots,\ell\}$) satisfy
\eqref{new:40-} and \eqref{eq:ds}.
By Theorem \ref{thm:7} these two conditions
are necessary for the existence of an $\mc{S}$-block measure $\mu\in \mc{L}_P$.

Since $P$ is assumed irreducible, the stochastic matrix $\La$ is irreducible also.  
By \blue{\eqref{eq:ds}} and Theorem \ref{birk}, we may find a measure $\rho\in\sL_\La$
supported on a subset of the set of permutations of $I$, and we let $\Pi$ be a random permutation with law $\rho$. 
(Note that $\rho$ is not generally unique.)
Let $i \in S_r$ and $j \in S_s$.  In order for our forthcoming $\mu$ to be consistent with $P$ (recall  Definition \ref{eq:4}) we need that
\begin{align}
\sum_{f: i\mapsto j} \mu(\{f\})  = p_{i,j}.\label{want}
\end{align}
In order for $i$ to be mapped to $j$, it is necessary that $\Pi(r)=s$; the last occurs with 
probability $\lambda_{r,s}$.  
 Conditional on $\Pi$, we shall then map the states independently in such a way as to obtain \eqref{want}.    

More precisely,
conditional on $\Pi$, let  $Z=(Z_i: i \in S)$ be independent random variables such that
\begin{equation}\label{new6}
\P(Z_i=j\mid \Pi) =
\begin{cases} p_{i,j}/\lambda_{r,s} &\text{if } S_r \ni i,\ S_s \ni j,\  \Pi(r)=s,\\
0 &\text{otherwise},
\end{cases}
\end{equation}
and let $\mu$ be the law of $Z$. Thus, $\mu\in M(\sF_S)$ is given by
\begin{equation}\label{new7}
\mu(\{f\}) = \E\left[\prod_{i\in S} \P(Z_i=f(i)\mid\Pi)\right], 
\end{equation}
where the expectation is over the random permutation $\Pi$. 
We call $\mu$ the
$(P,\sS,\rho)$-\emph{product measure}.

Next we check that $\mu\in \mc{L}_P$. 
Let $i\in S_r$ and $j\in S_s$. By \eqref{new6}--\eqref{new7},
$$
\sum_{f: i\mapsto j} \mu(\{f\}) =\P(Z_i=j)=\lam_{r,s} \cdot \dfrac{p_{i,j}}{\lam_{r,s}} = p_{i,j},
$$
as required. By the definition of $\mu$, no two states in different sets of the partition $\sS$ may coalesce.  
Thus, to determine whether the $(P,\sS,\rho)$-product measure $\mu\in \mc{L}_P$ is an $\mc{S}$-block measure  
it remains to check whether $k(\mu)=\ell=|\mc{S}|$, which is to say that, for all $r$, all states in $S_r$ coalesce
(recall \eqref{new100}).

Recall that $\mc{C}$ denotes the (possibly random) set of coalescing classes, 
and that  $i\sim j$ if there exists $C\in \mc{C}$ such that $i,j\in C$.

\begin{theorem}\label{prop:4}
Let $P\in\sP_S$, and let $\sS = \{S_r: r\in I\}$  be  a partition of $S$
with index set
$I=\{1,2,\dots,\ell\}$. 
Assume \eqref{new:40-} and \eqref{eq:ds} hold.
The $(P,\sS,\rho)$-product measure $\mu$ satisfies $k(\mu)=\ell$ if
and only if\/\footnote{One may work with any given value of $r$ in \eqref{new:31++},
and we have chosen $r=1$ for concreteness. 
Condition \eqref{new:31++} is similar to the sufficient condition of \cite[Thm 6]{Tak} for 
the quenched ergodicity of a Markov chain with random transition matrices.
Theorem \ref{prop:4} corrects an error in \cite[Thm 6]{ncftp}.}
\begin{equation}\label{new:31++}
\P(i \sim j)>0, \qq i,j\in S_1.
\end{equation}
\end{theorem}

\begin{proof}
Suppose the conditions of the theorem hold,
and also \eqref{new:31++}. 
Let $r\in I$. We will show that
\begin{equation}\label{new:43a}
\P(i\sim j)=1, \qq i,j\in S_r.
\end{equation}
The claim $k(\mu)=\ell$ follows since there are finitely many pairs $i,j\in S_r$ and indices $r$.

Recall from \eqref{eq:co} that
$T_{i,j}=\inf\{t\ge 0: X_t^i=X_t^j\}$,
so that $i\sim j$ if and only if $T_{i,j}<\oo$.

We first prove \eqref{new:43a} with $r=1$.
Since $|S_1|<\oo$, by \eqref{new:31++} there exists $\eps>0$,  and $M<\oo$ such that
\begin{equation}\label{new:32+}
\P(T_{i,j} \le M)>\eps, \qq i,j\in S_1.
\end{equation}

Let $i,j\in S_1$. Either $T_{i,j} \le M$ or not. 
Assume that $T_{i,j}>M$. We continue the two chains from time $M$ with initial states $X_M^i$ and $X_M^j$
until the next epoch ($M+K$, say) 
at which these two processes lie in $S_1$; note that $\P(K<\oo)=1$  
since $\La$  is irreducible.
Having arrived back in $S_1$, we apply the argument above to deduce that coalescence occurs by time $2M+K$ with 
(conditional) probability at least $\eps$.

It follows that
\begin{align*}
\P(T_{i,j} > 2M+K) &=\P(T_{i,j} > 2M+K\mid T_{i,j} > M) \P(T_{i,j}>M)\\
&\le (1-\eps)^2.
\end{align*}
By iteration of this argument, $\P(T_{i,j}<\oo)=1$, and  \eqref{new:43a} follows with $r=1$.  

For $r\ne 1$ and $i,j\in S_r$, since $\Lambda$ is irreducible, a.s.~$S_r$ is eventually mapped to $S_1$.  
At this point, $i$ and $j$ have both been mapped to elements of $S_1$, so (if they have not already done so) 
they will almost surely coalesce by \eqref{new:43a} with $r=1$.  This verifies \eqref{new:43a} 
for general $r$ as required.

Conversely, if there exist 
$i,j\in S_1$ such that $\P(i\sim j)=0$ 
then it is necessarily the case that $k(\mu)\ge \ell+1$.
\end{proof}

\begin{example}\label{ex10}
Here is an illustration of Theorem \ref{prop:4}.
As there (with $\ell=2$ for simplicity) we let $P\in\sP_S$, $\sS = \{S_1,S_2\}$, and we assume that 
\eqref{new:40-} and \blue{\eqref{eq:ds}} hold.

For $r=1,2$ we write $S_r=\{x_{r,j}: 1\le j \le m_r\}$.
Ordering the elements of $S$ as $(x_{1,1}.\dots, x_{1,m_1}, x_{2,1},\dots, x_{2,m_2})$, we may express 
the $(m_1+m_2)\times (m_1+m_2)$ matrix $P$ in the form
$$
P=
\begin{pmatrix} A & B\\ C & D
\end{pmatrix},
$$
where $A$ is an $m_1 \times m_1$ matrix and $D$ is an $m_2\times m_2$ matrix.

The $2\times 2$ matrix $\La$ is doubly stochastic and irreducible, 
and may (by Theorem \ref{birk}) be expressed in the form
\begin{equation}\label{eq:unique}
\La = \begin{pmatrix} \lam_{1,1} & \lam_{1,2} \\ \lam_{2,1} & \lam_{2,2} \end{pmatrix}
= \alpha I + (1-\alpha) \oI
\end{equation}
where $\alpha=\lam_{1,1}=\lam_{2,2}$ and 
$$
I=\begin{pmatrix} 1 & 0\\ 0 & 1\end{pmatrix},
\qq 
\oI=\begin{pmatrix} 0 & 1\\ 1 & 0\end{pmatrix}.
$$ 
In other words, the measure $\rho$ (which is unique since $\Lambda$ is a $2\times 2$ matrix) puts mass $\alpha$ on the identity permutation, and mass $1-\alpha$ on the \lq interchange' permutation.  
We let $(\Pi_i:i \in \N)$ be independent permutations 
with common law $\rho$, and let $\mu$ be the $(P,\mc{S},\rho)$-product measure.

We have shown that
\[\Lambda=\begin{pmatrix}
\alpha & 1-\alpha\\
1-\alpha & \alpha
\end{pmatrix}.\]
Thus, the row sums of $A$ and $D$ (\resp, $B$ and $C$) are all $\alpha$ (\resp, $1-\alpha$) by 
 \eqref{new:39+} and \eqref{new:40-}).
Therefore, $P$ is a mixture of two stochastic matrices $P_1$ and $P_2$,
$$
P =  \alpha P_1 + (1-\alpha) P_2,
$$
where
$$
P_1 = \frac1\alpha \begin{pmatrix} A & 0\\ 0 & D
\end{pmatrix},\qq
P_2 =\frac1{1-\alpha} \begin{pmatrix} 0 & B\\ C & 0
\end{pmatrix}.
$$

\emph{Suppose $\alpha>0$ (we have $\alpha<1$ by irreducibility), and that $A':=A/\alpha$ is irreducible and aperiodic.}
We claim that $k(\mu)=2$. This may be shown as follows.

Let $W=(W_n: n\ge 0)$ be a Markov chain on $S_1$ with transition matrix $A'$, and let $i,j\in S_1$.
By Theorem \ref{thm:indep_copies_aperiodic}, 
if $W^i$ and $W^j$ are two independent versions of $W$, starting \resp\  at $i$ and $j$, then 
there exists $M<\oo$ such that 
$$
\P(W_m^i=W_m^j\text { for some }m\le M)  \ge \tfrac12.
$$
It follows that
\begin{align*}
\P(T_{i,j}<\oo) &\ge \P(T_{i,j}<\oo,\ \Pi_m=I\text { for } m=1,2,\dots, M)\\
&\ge\alpha^M \P(W_m^i=W_m^j\text { for some }m\le M)  \ge \tfrac12\alpha^m.
\end{align*}
Equation \eqref{new:31++} follows, and hence $k(\mu)=2$ by Theorem \ref{prop:4}.
\END
\end{example}

Finally, in preparation for the next section we state a result from \cite{ncftp} which concerns the transition matrix 
\begin{equation*}
P_n=\begin{pmatrix}
n^{-1} & n^{-1} & \cdots &  n^{-1} \\
n^{-1} & n^{-1} & \cdots & n^{-1} \\
\vdots & \vdots & \ddots & \vdots\\
n^{-1} & n^{-1} & \cdots & n^{-1} 
\end{pmatrix}
\end{equation*}
on the state space $S=\{1,2,\dots,n\}$ with equal entries. 

\begin{theorem}[{\cite[Thm 7]{ncftp}}]
\label{thm:blockPn}
For $n \ge 2$ there exists a block measure $\mu\in\sL(P_n)$ with
$k(\mu)=\ell$ if and only if $\ell \mid n$. In particular,  $K(P_n)\supseteq \{\ell: \ell\mid n\}$.
For $n \ge 3$, we have $n-1\notin K(P_n)$.
\end{theorem} 

\section{Existence of non-block measures}\label{sec:exist-non}

Many of our measures $\mu\in\sL_P$ so far have been block measures, though we encountered a certain
$\mu$ in Example \ref{ex:7} which (by Theorem \ref{thm:coal}) is not a block measure.  
By Theorem \ref{thm:blockPn}, $n-1\notin K(P_n)$, whence $K(P_n)=\{\ell:\ell\mid n\}$ when $n\le 4$.  
We do not know whether $K(P_n)=\{\ell:\ell\mid n\}$ when $n> 4$.
In this section we shall construct a family of non-block measures for $P_n$.

\begin{theorem}\label{thm:6}
Let $n \ge 4$ and suppose $\ell\ne 1$ and $\ell \mid n$.
There exists a non-block measure $\mu\in\sL_{P_n}$ with
$k(\mu)=\ell$.
\end{theorem}

Before embarking on the proof of Theorem \ref{thm:6},
we present an illustrative  example.




\begin{example}
\label{exa:n6b3}
Let $n=6$, $\ell=2$, and $b=n/\ell=3$.  
Let $S_1=\{1,2\}$, $S_2=\{3,4\}$, and $S_3=\{5,6\}$.  Let $\what\sF\subset \mc{F}_S$ consist of the six functions
\begin{align*}
f_{1,1}&=(121212),
&f_{1,2}&=(212121),\\
f_{2,1}&=(343443),
&f_{2,2}&=(434334),\\
f_{3,1}&=(566565),
&f_{3,2}&=(655656).
\end{align*}
We make the following observations.
\begin{romlist}
\item
 For each $u,v\in \{1,2,\dots, 6\}$, $f(u)=v$ for exactly one $f\in \what\sF$.  Thus, taking $\mu$ to be uniform on $ \what\sF$,
  we 
 obtain that $\mu \in \mc{L}_{P_6}$.  
\item
Let $ f\in\what\sF$. For each $i$ there exists $j$ such that $f$  maps $S_i$ onto $S_j$.   Thus no pair of states in any $S_i$ will ever coalesce.
\item 
The functions $f_{1,j}$ yield  the immediately coalescing classes $\{1,3,5\}$, $\{2,4,6\}$.  The functions $f_{2,j}$ yield  the immediately coalescing classes $\{1,3,6\}$, $\{2,4,5\}$.  The functions $f_{3,j}$ yield  the immediately coalescing 
classes $\{1,4,6\}$, $\{2,3,5\}$.  All coalescence takes place on the first step, and depending on which function is chosen first (i.e., which function $F_1$ is) we obtain different coalescing classes.\END
\end{romlist}

\end{example}
 
\begin{proof}[Proof of Theorem \ref{thm:6}]
Let $n, \ell$ be as in the statement of the theorem. We will imitate Example 
\ref{exa:n6b3}, and the reader may wish to refer back to that example.   Let $b=n/\ell$ and   
\[
S_r=\bigl\{(r-1)\ell+1,(r-1)\ell+2,\dots, r\ell\bigr\}, \qq r=1,2,\dots,b,
\]
 noting that $|S_r|=\ell$ for each $r$.  
 When considering
$S_r$ as a vector rather than as a set, we order its elements in increasing order.

We will construct a set $\what\sF$ comprising $n$ ($=b\ell$)  functions $f_{i,j}$ with
$i \in \{1,2,\dots, b\}$, $j \in \{1,2,\dots, \ell\}$; 
$\mu$ will be the uniform measure on $\what\sF$. Before
giving a formal definition of $\what\sF$, we summarise its main properties as follows:
\begin{romlist}
\item for every $u,v \in \{1,2,\dots, n\}$ there is exactly one $f\in \what\sF$ such that $f(u)=v$, 
\item for each $f\in \what\sF$, 
 there is a block $S_{f}$ such that every block  $S_r$ is mapped by $f$ \emph{onto} $S_{f}$
 (viewed as sets),
 \item the elements of each $S_r$ (viewed as a vector) are mapped to a certain permutation $\pi^r S_f$
 of $S_f$. 
\end{romlist} 

We turn now to a formal definition, beginning with some notation.
Following Definition \ref{def:fn}, each $f_{i,j}$ can be expressed in the form $(x_1x_2\cdots x_n)$, which is to say that $f_{i,j}(u)=x_u$.  It is convenient to break such a vector into
consecutive subsequences of length $\ell$,
and we do this by adding vertical bars; thus, for example, we write $(x_1x_2\cdots x_n)$ as
\[
\big(x_1\cdots x_\ell\big|x_{\ell+1}\cdots x_{2\ell}\big|\cdots \big|x_{(b-1)\ell+1}\cdots x_{b\ell}\big),
\]
and we term the $b$ subsequences  therein \lq image blocks'.
For each  $f_{i,j}$, the image elements $x_1,\dots,x_\ell$ are distinct;
furthermore, when viewed as sets, we have that 
$\{x_{(i-1)\ell+1},\dots, x_{i\ell}\}=\{x_1,\dots,x_\ell\}$ for each $i$, 
so that each subsequent image block  of each $f_{i,j}$ 
is a permutation of the first image block of $f_{i,j}$.  

We explain next the action of the $f_{i,j}$.
Let $\rho$ denote the rotation permutation $i_1i_2\cdots i_\ell\mapsto i_2i_3\cdots i_\ell i_1$.  
We will express each $f_{i,j}$ in terms of a pair $(h_{i,j}, \bpi^i)$ where
\begin{letlist}
\item $h_{i,1} = S_i$ and  $h_{i,j} = \rho^{j-1}S_i$ (with $S_i$ considered as a vector),
\item 
for each $i=1,2,\dots,b$,
there exists a vector $\bpi^i= (\pi^i_2,\pi^i_3,\dots, \pi^i_b)$ of $b-1$ permutations 
(not necessarily distinct) of $\ell$ symbols,
\item
we set 
\begin{equation}\label{eq:defij}
f_{i,j} = \big( h_{i,j} \big| \pi^i_2 h_{i,j} \big|\cdots \big| \pi^i_b h_{i,j}
\big).
\end{equation}
\end{letlist}
We discuss next how the $\bpi^i$ are defined, beginning with the simplest case $i=1$.  

Let each  $\pi^1_j$ be the identity 
permutation, so that, by \eqref{eq:defij},
\begin{align*}
f_{1,1}= \big( h_{1,1} \big| \pi^1_2 h_{1,1} \big|\cdots \big| \pi^1_b h_{1,1}
\big)
=\big(12\cdots \ell\big|12\cdots \ell\big|\cdots \big|12\cdots\ell\big).
\end{align*}
More generally, we let
\begin{equation}\label{new54}
f_{1,j} = \big( h_{1,j} \big| \pi^1_2 h_{1,j} \big|\cdots \big| \pi^1_b h_{1,j}
\big),
\end{equation}
where
\begin{equation*}
h_{1,j} = \pi_m^1 h_{1,j} 
=\big(j(j+1) \cdots (\ell-1)\ell 1 2 \cdots (j-1)\big), \qq m\ge 2.
\end{equation*}

We discuss next the case of $f_{i,j}$ with $i\ge 2$.  
Then  $h_{i,1}=S_i$ and $h_{i,j}=\rho^{(j-1)}S_i$, and it remains to choose 
a suitable vector $\bpi^i$ of permutations.
For $i=2,3,\dots, b$, let $\bpi^i = (\pi^i_2,\dots, \pi^i_b)$ be an ordered set of permutations 
satisfying $(\pi^i_2,\dots, \pi^i_b)\ne (\pi^{i'}_2,\dots, \pi^{i'}_b)$ for each $i'<i$ (in other words, all the ordered sets of permutations will be distinct as ordered sets).  
We can always find distinct $\bs{\pi}^1,\bs{\pi}^2,\dots, \bs{\pi}^b$ since there are $(\ell!)^{b-1}$ distinct such vectors and $(\ell!)^{b-1}\ge 2^{b-1}\ge b$ (here we have used the fact that $b\ge 2$).
The function $f_{i,j}$ is given by \eqref{eq:defij}.

By construction, $\what\sF$ has properties (i)--(iii) above.
Let $\mu$ be the uniform probability measure on $\what{\sF}$.
By (i),
$$
\mu(\{f:f(u)=v\})=\frac 1n, \qquad  u,v \in S,
$$
whence $\mu \in \sL(P_n)$.  

We turn to the issue of coalescence. By (ii)--(iii) above, 
coalescence occurs at the first stage and not subsequently.  
Since there are at least two distinct $\bs{\pi}^i$, the coalescing classes are random (indeed the number of possibilities for the set of coalescing classes is the number of distinct $\bs{\pi}^i$, $i=1,2,\dots, b$).
Since mappings between blocks are surjections, there are exactly $\ell$ coalescing classes 
$C_1,C_2,\dots, C_\ell$, and each such $C_r$ is a transversal of $\sS$.
Hence $k(\mu)=\ell$.
\end{proof}

\begin{remark}
In the above proof, we have required that the $\bpi^i$ be distinct. It suffices 
for the proof that at least two of them are distinct. Suppose, on the contrary,
that $\bs{\pi}^j=\bs{\pi}^1$ for all $j$.   
Then the set of coalescing classes is deterministic, and Theorem \ref{thm:coal} applies.  
In particular since $\bs{\pi}^1$ is the identity permutation we may  see that 
$$
C_j=\{j,\ell+j,\dots, (b-1)\ell+j\}, \q j=1,2,\dots, \ell.
$$
We deduce that $k(\mu)=\ell$, and that $\mu$ is a block measure with blocks $C_1,C_2,\dots, C_\ell$.  
\END\end{remark}

\section{Functions that generate transition matrices}\label{sec:whatF}

In this final section we pose an inverse question. As usual, the state space is 
$S=\{1,2,\dots,n\}$. Recall the set $\sP_S$ of irreducible  transition matrices on $S$ and
the set $\sF_S$ of functions from $S$ to $S$.
Given a subset $\mc{G}\subseteq\mc{F}_S$, let 
\[
\mc{P}(\mc{G})=\{P\in\sP_S: \exists \mu\in \mc{L}_P \text{ with }\supp(\mu)\subseteq \mc{G}\}.
\] 
In other words, $\mc{P}(\mc{G})$ is the set of (irreducible) stochastic matrices that can be obtained by just using functions in $\mc{G}$.

\begin{question}\label{qn:2}
For given $\sG\subseteq\sF_S$, what can be said about the set $\sP(\sG)$?
\end{question}

\begin{example}
\label{ex:fxy}
For distinct $x,y\in S$ let 
\[
f_{x,y}(v)=\begin{cases}
y & \text{ if }v=x,\\
v & \text{ otherwise}.
\end{cases}
\]
This is \lq almost' the identity function (except  that $x\mapsto y$).  
Let $\mc{G}=\{f_{x,y}:x,y\in S, \,  x\ne y\}$, so that $\sG$  contains 
$n(n-1)$ functions.  Let  $\bs{\alpha}=(\alpha_{x,y}:x\ne y)$ be a vector
of non-negative reals 
satisfying 
\[
\sum_{x,y:\,x\ne y}\alpha_{x,y}=1.
\]
 Let $\mu_{\bs{\alpha}}$ be the  probability measure on $\sG$
 that selects $f_{x,y}$ with probability $\alpha_{x,y}$.  
 The corresponding $P=P_{\bs{\alpha}}=(p_{i,j})$ satisfies
\[
p_{i,j}=\mu\bigl(\{f:f(i)=j\}\bigr)=\mu(\{f_{i,j}\})=\alpha_{i,j},\qq i\ne j,
\]
while 
\[p_{i,i}=\mu\bigl(\{f:f(i)=i\}\bigr)=1-\sum_{j:\,j\ne i}\alpha_{i,j}.\]
Therefore, $\mc{P}(\mc{G})$ is the set of stochastic matrices whose off-diagonal 
elements sum to $1$.\END
\end{example}
Some further very special cases admit precise answers to Question \ref{qn:2}.  
Let $\mc{D}_S\subset \mc{P}_S$ denote the set of doubly stochastic matrices, and let
 $\sFp_S\subset \mc{F}_S$ denote the set of permutations of $S$.   

\begin{proposition}\label{prop:DP}
\mbox{\hfil}
\begin{letlist}

\item 
$\mc{P}(\mc{G})=\mc{D}_S$ if and only if $\mc{G}= \sFp_S$.  
\item 
$\mc{P}(\mc{G})=\mc{P}_S$ if and only if $\mc{G}= \mc{F}_S$.

\end{letlist}
\end{proposition}

\begin{proof}
(a) This is a consequence of Theorem \ref{birk}.

(b)
Suppose first that $\mc{G}=
\mc{F}_S$. For $P\in \mc{P}_S$  one can construct its independence coupling (recall Example \ref{exa:indep_grand}).  Therefore $\mc{P}(\mc{G})=\mc{P}_S$.

Suppose conversely that $\sG$ is such that $\mc{P}(\mc{G})=\mc{P}_S$.  
We claim that for every $f\in \mc{F}_S$ there exists $P\in \mc{P}_S$ (hence $P \in \mc{P}(\mc{G})$) such that $f\in \supp(\mu)$ for every $\mu \in \mc{L}_P$.  This implies that $f\in \mc{G}$ and therefore  $\mc{G}= \mc{F}_S$ as required.  

 It remains to verify the above claim.   Let $f=(j_1j_2\dots j_n)\in \mc{F}_S$
 (recall Definition \ref{def:fn}).   
 There exists $P_f=(p_{i,j})\in \mc{P}_S$ such that  $p_{i,j_i}>1-1/n$ for every $i\in S$;
 the remaining terms $p_{i,j}$, $j\ne j_i$, are assumed strictly positive,
 thus implying that $P_f$ is irreducible.  
 For $\mu \in \mc{L}_{P_f}$, we have 
 \[
 \mu\bigl(\{g\in\sF_S:g(i)=j_i\}\bigr)>1-\frac 1n, \qq i=1,2,\dots,n,
 \] 
 whence 
\[
\mu(f) = \mu\left(\bigcap_{i\in S}\{g:g(i)=j_i\}\right)>0.
\]
Therefore, $f\in \supp(\mu)$. 
\end{proof}

We recall the random transition matrix of Remark \ref{rem:Q}, with law $\Leb$.

\begin{question}\label{q:4}
\label{qn:LebG}
For given $\mc{G}\subset \mc{F}_S$, what can be said about $\Leb(\mc{P}(\mc{G}))$?
\end{question}

This section closes with some partial answers to this question.

\emph{Firstly}, in  the proof of Proposition \ref{prop:DP}(b) it is shown that, if $f\notin \mc{G}$, then 
any matrix with $p_{i,f(i)}>1-n^{-1}$ for all $i \in S$ does not belong to $\mc{P}(\mc{G})$.  
It follows that $\Leb(\mc{P}(\mc{G}))<1$ whenever $\mc{G}\ne \mc{F}_S$.  

\emph{Secondly}, in Example \ref{ex:fxy}, we have $\Leb(\mc{P}(\mc{G}))=0$ since $\sG$ 
is the set of stochastic matrices with off-diagonal elements summing to $1$.

\emph{Thirdly}, in the case $|S|=2$ there are just 4 functions, namely
 \[(11), (12), (21), (22).\]  
It is elementary that if $\mc{G}$ is missing two of them then $\Leb(\sP(\mc{G}))=0$.  
The following lemma shows that (when $n=2$) removing a single function one retains positive measure.

\begin{proposition}
If $|S|=2$ and $f\in\mc{F}_S$, then  $\Leb(\mc{P}(\mc{F}_S\setminus\{f\}))=\frac12$.
\end{proposition}

\begin{proof}
Given a  probability mass function $\bs{\alpha}$ on $\mc{F}_S$, the resulting matrix $P$ is given by 
\begin{equation}
P = \begin{pmatrix}
\alpha_{(11)}+\alpha_{(12)}& \alpha_{(21)}+\alpha_{(22)}\\
\alpha_{(11)}+\alpha_{(21)}& \alpha_{(12)}+\alpha_{(22)}
\end{pmatrix}.\label{2by2}
\end{equation}

Henceforth, fix $f=(uv)$ and write $u'\ne u$ and $v'\ne v$. 
Let $\mc{P}_{uv}$ denote the set of $P\in \sP_S$ for which $p_{2,v}<p_{1,u'}$.
Since $p_{2,v}$ and $p_{1,u'}$ are independent with the same probability density function
(under $\Leb$), we have that $\Leb(\mc{P}_{uv})=\frac12$. We claim that
\begin{equation}\label{eq:n=2}
\mc{P}(\mc{F}_S\setminus \{f\})=\mc{P}_{uv},
\end{equation}
and the proposition will follow immediately. It remains to prove \eqref{eq:n=2}.

Let $P\in\sP(\sF_S\sm\{f\})$. There exists a mass function $\bal$ on $\sF_S$ satisfying $\al_{(uv)}=0$
such that \eqref{2by2} holds; note that $\al_{(ab)}>0$ for $(ab)\ne (uv)$, $\Leb$-a.s., and
we will assume this. By \eqref{2by2},
\begin{equation*}
p_{2,v}=\al_{(u'v)} 
<\alpha_{(u'v)}+\alpha_{(u'v')} =p_{1,u'}.
\end{equation*}
Since $P$ is stochastic, we have also that $p_{2,v'}>p_{1,u}$,
and hence $P\in \mc{P}_{uv}$ as required.  

Conversely, let $P\in \mc{P}_{uv}$, so that $p_{2,v}<p_{1,u'}$.   Set 
\[\alpha_{(uv)}=0, \q \alpha_{(uv')}=p_{1,u}, \quad  \alpha_{(u'v)}=p_{2,v}, \quad \alpha_{(u'v')}=1-(p_{1,u}+p_{2,v}),\]
and note that 
\[1-\al_{(u'v')} =p_{1,u}+p_{2,v}<p_{1,u}+p_{1,u'}=1.
\]
The stochastic matrix $Q=(q_{i,j})$ 
corresponding to $\bs{\alpha}$ (according to \eqref{2by2}) satisfies 
\begin{align*}
q_{1,u}&=\alpha_{(uv)}+\alpha_{(uv')}=\alpha_{(uv')}=p_{1,u},\\
q_{2,v}&=\alpha_{(uv)}+\alpha_{(u'v)}=\alpha_{(u'v)}=p_{2,v},
\end{align*}
whence $Q=P$.  Therefore, $\mc{P}_{uv}\subseteq\mc{P}(\mc{F}_S\setminus \{f\})$ and the proof
of \eqref{eq:n=2} is complete.
\end{proof}

\emph{Finally}, for any $\vep>0$, when $|S|$ is large there exist sets $\mc{G}\subseteq\sF_S$ 
with size satisfying $|\mc{G}|>(1-\vep)|\mc{F}_S|$ for which $\mc{P}(\mc{G})$ 
has $\Leb$-measure zero.  Here are two examples.

\begin{proposition}\label{prop:5}
We have $\Leb(\sP(\sG))=0$ in the following two cases.
\begin{letlist}
\item
Let $i,j\in S$ and $\sG=\sG_{i,j} =\{f\in \mc{F}_S:f(i)\ne j\}$.

\item
Let $i_1, i_2, j_1,j_2\in S$ satisfy $i_1\ne i_2$, and let $\sG$ be the set of all $f\in\sF_S$ such that 
\[
f(i_1)=j_1 \q\text{if and only if}\q f(i_2)=j_2.
\]
\end{letlist}
\end{proposition}

Let $n=|S|$.
The cardinality of $\mc G$ in part (a) is $|\mc{F}_S|(1-n^{-1})$, and
in part (b) is $|\sF_S|(1-2n^{-1}+\o(n^{-1}))$.

 \begin{proof}
 (a)  For any $\mu$ supported on $\mc{G}_{i,j}$ we have 
$\mu(\{f:f(i)=j\})=0$, whence every $P=(p_{i,j}) \in \mc{P}(\mc{G})$ has $p_{i,j}=0$.
The probability of such $P$ satisfies $\Leb(\mc{P}(\mc{G}))=0$.  

(b)  For any $\mu$ supported on such $\mc{G}$ we have 
\[\mu(\{f:f(i_1)=j_1\})= \mu(\{f:f(i_2)=j_2\}),\]
 so every $P=(p_{i,j}) \in \mc{P}(\mc{G})$ has $p_{i_1,j_1}=p_{i_2,j_2}$.  The claim follows.
 \end{proof}
 
 \begin{example}
Let $n=3$, so that $|\mc{F}_S|=27$. Let $\mc{G}$ be the following set of 15 functions:
\begin{align*}
\mc{G}&=\{(111),(311),(121),(321),\bs{(231)},\\
&\qquad(112),(312),(122),(322),\bs{(232)},\\
&\qquad(113),(313),(123),(323),\bs{(233)}\}.
\end{align*}
Note that $1\mapsto 2$ if and only if $2\mapsto 3$ (the relevant functions  are in bold above) 
so, by Proposition \ref{prop:5}(b), $\mc{P}(\mc{G})$ has zero 
$\Leb$-measure.\END
 \end{example}

\section*{Acknowledgements}
The authors are grateful to Tabet Aoul Alaa Eddine Islem for indicating an error in an 
earlier version of Theorem \ref{thm:7} as published in \cite{ncftp}. They thank
two referees for their useful comments. 


\begin{thebibliography}{10}

\bibitem{AHMWW}
O.~Angel, A.~E. Holroyd, J.~Martin, D.~B. Wilson, and P.~Winkler,
  \emph{Avoidance coupling}, Electron. Commun. Probab. \textbf{18} (2013), no.
  58, 13.

\bibitem{BP}
E.~Bates and M.~Podder, \emph{Avoidance couplings on non‐complete graphs},
  Random Struct. Alg. \textbf{59} (2021), 25--52.

\bibitem{Birkh}
G.~Birkhoff, \emph{Three observations on linear algebra}, Univ. Nac.
  Tucum\'{a}n. Revista A. \textbf{5} (1946), 147--151.

\bibitem{Bremaud}
P.~Br\'emaud, \emph{{Markov Chains}}, 2nd ed., Springer Nature, Cham, 2020.

\bibitem{CL}
J.-F. Chamayou and G.~Letac, \emph{Explicit stationary distributions for
  compositions of random functions and products of random matrices}, J.
  Theoret. Probab. \textbf{4} (1991), 3--36.

\bibitem{DP}
P.~Diaconis and D.~Freedman, \emph{Iterated random functions}, SIAM Rev.
  \textbf{41} (1999), 45--76.

\bibitem{Doeblin}
W.~Doeblin, \emph{{Expos\'e de la th\'eorie des cha\^ines simples constantes de
  Markoff \`a un nombre fini d'\'etats}}, Revue Math. de l'Union
  Interbalkanique \textbf{2} (1938), 77--105.

\bibitem{geiger}
B.~C. Geiger and C.~Temmel, \emph{Lumpings of {M}arkov chains, entropy rate
  preservation, and higher-order lumpability}, J. Appl. Probab. \textbf{51}
  (2014), 1114--1132.

\bibitem{gdpy}
K.~Georgiou, G.~N. Domazakis, D.~Pappas, and A.~N. Yannacopoulos, \emph{Markov
  chain lumpability and applications to credit risk modelling in compliance
  with the {I}nternational {F}inancial {R}eporting {S}tandard 9 framework},
  European J. Oper. Res. \textbf{292} (2021), 1146--1164.

\bibitem{ncftp}
G.~R. Grimmett and M.~Holmes, \emph{Non-coupling from the past}, {In and Out of
  Equilibrium 3. {C}elebrating {V}ladas {S}idoravicius}, Progr. Probab.,
  vol.~77, Birkh\"auser/Springer, Cham, 2021, pp.~471--485.

\bibitem{GS}
G.~R. Grimmett and D.~R. Stirzaker, \emph{{Probability and Random Processes}},
  4th ed., Oxford University Press, Oxford, 2020.

\bibitem{GW}
G.~R. Grimmett and D.~Welsh, \emph{{Probability---an Introduction}}, 2nd ed.,
  Oxford University Press, Oxford, 2014.

\bibitem{KS}
J.~G. Kemeny and J.~L. Snell, \emph{{Finite Markov Chains}}, {Van Nostrand},
  1963.

\bibitem{Letac}
G.~Letac, \emph{A contraction principle for certain {M}arkov chains and its
  applications}, Random matrices and their applications ({B}runswick, {M}aine,
  1984), Contemp. Math., vol.~50, Amer. Math. Soc., Providence, RI, 1986,
  pp.~263--273.

\bibitem{marin}
A.~Marin, C.~Piazza, and S.~Rossi, \emph{Proportional lumpability and
  proportional bisimilarity}, Acta Inform. \textbf{59} (2022), 211--244.

\bibitem{vN}
J.~von Neumann, \emph{A certain zero-sum two-person game equivalent to the
  optimal assignment problem}, Contributions to the theory of games, vol. 2,
  Annals of Mathematics Studies, no. 28, Princeton University Press, Princeton,
  NJ, 1953, pp.~5--12.

\bibitem{Norris}
J.~R. Norris, \emph{{Markov Chains}}, Cambridge Series in Statistical and
  Probabilistic Mathematics, vol.~2, Cambridge University Press, Cambridge,
  1998.

\bibitem{PW2}
J.~G. Propp and D.~B. Wilson, \emph{Exact sampling with coupled {M}arkov chains
  and applications to statistical mechanics}, Random Structures Algorithms
  \textbf{9} (1996), 223--252, Proceedings of the {S}eventh {I}nternational
  {C}onference on {R}andom {S}tructures and {A}lgorithms ({A}tlanta, {GA},
  1995).

\bibitem{PW3}
\bysame, \emph{How to get a perfectly random sample from a generic {M}arkov
  chain and generate a random spanning tree of a directed graph}, J. Algorithms
  \textbf{27} (1998), 170--217, 7th Annual ACM-SIAM Symposium on Discrete
  Algorithms (Atlanta, GA, 1996).

\bibitem{Tak}
Y.~Takahashi, \emph{Markov chains with random transition matrices}, K\= odai
  Mathematical Seminar Reports \textbf{21} (1969), 426--447.

\bibitem{WP1}
D.~B. Wilson and J.~G. Propp, \emph{How to get an exact sample from a generic
  {M}arkov chain and sample a random spanning tree from a directed graph, both
  within the cover time}, Proceedings of the {S}eventh {A}nnual {ACM}-{SIAM}
  {S}ymposium on {D}iscrete {A}lgorithms ({A}tlanta, {GA}, 1996), ACM, New
  York, 1996, pp.~448--457.

\end{thebibliography}
\providecommand{\bysame}{\leavevmode\hbox to3em{\hrulefill}\thinspace}
\providecommand{\MR}{\relax\ifhmode\unskip\space\fi MR }
\providecommand{\MRhref}[2]{%
  \href{http://www.ams.org/mathscinet-getitem?mr=#1}{#2}
}
\providecommand{\href}[2]{#2}

\end{document}